\documentclass[11pt]{amsart}

\usepackage{amssymb}
\usepackage{amsfonts}
\usepackage{amsthm}
\usepackage{graphicx}
\usepackage{color}
\usepackage{url}
\usepackage{microtype}
\usepackage{tikz}

\usepackage{hyperref}

\definecolor{darkblue}{RGB}{0,0,160}
\hypersetup{
colorlinks,%
citecolor=black,%
filecolor=black,%
linkcolor=darkblue,%
urlcolor=darkblue
}

\usepackage[utf8]{inputenc}
\usepackage{todonotes}

\newtheorem{thm}{Theorem}[section]
\newtheorem{lemma}[thm]{Lemma}

\newtheorem{cor}[thm]{Corollary}
\newtheorem{prop}[thm]{Proposition}

\theoremstyle{definition}

\newtheorem{remark}[thm]{Remark}
\newtheorem{defn}[thm]{Definition}

\numberwithin{equation}{section}

\newcommand{\ring}[1]{\ensuremath{\mathbb{#1}}}

\newcommand\NN{\ring{N}}

\newcommand\QQ{\ring{Q}}
\newcommand\RR{\ring{R}}

\newcommand\ZZ{\ring{Z}}

\newcommand\cD{{\mathcal D}}

\newcommand\cN{{\mathcal N}}

\newcommand\cL{{\mathcal L}}
\newcommand\cM{{\mathcal M}}
\newcommand\cG{{\mathcal G}}
\newcommand\cP{{\mathcal P}}

\DeclareMathOperator\rank{rank} 
\newcommand{\inD}[1][\relax]{\def\argone{#1}\def\temprelax{\relax}
  \ifx\argone\temprelax\right.\else\,\middle|#1\right.{}\fi}

\newcommand{\fiberGraph}[2]{{#1}(#2)}

\newcommand{\diffGraph}[2]{\cG^{#1}_{#2}}

\newcommand{\fdim}[1]{\mathrm{fdim}(#1)}
\newcommand{\conv}[1]{\mathrm{conv}_{\QQ}(#1)}

\newcommand{\vertices}[1]{V(#1)}
\newcommand{\edges}[1]{E(#1)}

\begin{document}

\title{The fiber dimension of a graph}

\author{Tobias Windisch}
\address{Otto-von-Guericke Universität\\ Magdeburg, Germany} 
\email{windisch@ovgu.de}

\date{\today}

\makeatletter
  \@namedef{subjclassname@2010}{\textup{2010} Mathematics Subject Classification}
\makeatother



\begin{abstract}
Graphs on integer points of polytopes whose edges come from a set of
allowed differences are studied. It is shown that any simple graph can be
embedded in that way. The minimal dimension of such a representation
is the fiber dimension of the given graph.  The fiber dimension is
determined for various  classes of graphs and an upper bound in terms
of the chromatic number is stated.
\end{abstract}

\maketitle

\section{Introduction}

The study of geometric properties of graphs is a key ingredient in
understanding their algorithmic behaviour and combinatorial
structure~\cite{Thomassen1994,Lovasz2002}. In~\cite{Erdos1965}, the
dimension of a graph was introduced, which is the smallest $n\in\NN$
such that the graph can be embedded in $\RR^n$ with every edge having
unit length. 
Recently, isometrically embeddings of graphs into discrete objects
like hypercubes or lattices were paid a lot of attention in the
literature and lead to a lot of variations like the isometric
dimension~\cite{Fitzpatrick2000}, the lattice
dimension~\cite{Eppstein2005,Imrich2009}, or the Fibonacci
dimension~\cite{Cabello2011} of a graph.

In this paper, a new notion is added to the list of graph dimensions
(Remark~\ref{r:Dimension}).  Our concept has its origin in
\emph{algebraic statistics}, an emerging field which explores
statistical questions with algebraic
tools~\cite{Diaconis1998a,Drton2008,Hara2010,Aoki2012}. A main task
there is to construct connected graphs on integer points of polytopes
in order to draw samples by performing a random
walk~\cite[Chapter~5]{Sturmfels1996}. For a given polytope
$P\subset\QQ^d$ and a symmetric set $\cM\subset\ZZ^d$, a graph on
$P\cap\ZZ^d$ is given by connecting two nodes $u$ and $v$ by an edge
if $u-v\in\cM$. Graphs which can be obtained in that way are often
referred to as \emph{fiber graphs} in the literature and can be
understood as a discrete analogue of unit distance
graphs~\cite{Buckley1988}.

At first glance, it seems that fiber graphs are distinguished graphs
with their own rich structure. However, as it turns out, every graph
can be represented as a fiber graph
(Proposition~\ref{p:EveryGraphIsFiberGraph}). This motivates the
question for the smallest dimension in which a graph $G$ can be
represented as a fiber graph, the \emph{fiber dimension} of $G$
(Definition~\ref{d:FiberDimension}). We explore general properties of
this dimension and state upper bounds in terms of the chromatic number
(Theorem~\ref{t:ChromaticBound}) in the spirit of~\cite{Erdos1965}. We
then determine the fiber dimension for a variety of graphs. The fiber
dimension of a cycle of length $n$ depends on Euler's totient function
and we show that $\fdim{C_n}=1$ if and only if
$n\in\NN\setminus\{3,4,6\}$. Cycles whose length is one of the
exceptional cases $n\in\{3,4,6\}$ have fiber dimension $2$
(Proposition~\ref{p:Cycles}). Its proof uses the well-known fact that
Euler's totient function of $n\in\NN$ is $2$ if and only if
$n\in\{3,4,6\}$.  We also determine the fiber dimension of complete
graphs and show that it is logarithmic in the number of nodes
(Theorem~\ref{t:CompleteGraphs}). In the end, a connection to
\emph{distinct pair-sum polytopes}~\cite{Choi2002} is established and
it is shown how the fiber dimension leads to relations between the
number of vertices and the dimension of the ambient space of these
polytopes.

\subsection*{Conventions and Notations}
The natural numbers are denoted by $\NN=\{0,1,2,\dots\}$ and for
$n\in\NN$ with $n>0$, $[n]:=\{1,2,\dots,n\}$. All graphs
that appear in this paper are \emph{simple}, i.e. all edges are
undirected and they do not have loops or multiple edges. For any graph
$G$, the set of nodes is denoted by $\vertices{G}$ and its
\emph{chromatic number} is $\chi(G)$. The unit vectors of $\QQ^d$ are
denoted by $e_1,\dots,e_d$. For any $n\in\NN$, $K_n$ and $C_n$ denote
the complete graph and the cycle graph on $n$ nodes respectively.

\subsection*{Acknowledgements}
The author is thankful to Thomas Kahle for his comments on this
manuscript and to Benjamin Nill for pointing him to
\cite{Choi2002}. He gratefully acknowledges the support received
from the \href{https://www.studienstiftung.de}{German National
Academic Foundation}. 

\section{Fiber graphs}

A polytope $P\subset\QQ^d$ is a \emph{lattice polytope} if all its
vertices are in $\ZZ^d$. A finite set $\cM\subset\ZZ^d\setminus\{0\}$
is a \emph{set of moves} if $\cM=-\cM$ and if for all $\lambda\in\NN$
with $\lambda\ge 2$ and all $m\in\cM$, $\lambda\cdot m\not\in\cM$. 

\begin{defn}\label{d:FiberGraph}
Let $P\subset\QQ^{d}$ be a lattice polytope and let
$\cM\subset\ZZ^{d}$ be a set of moves. The \emph{fiber graph}
$\fiberGraph{P}{\cM}$ is the graph on $P\cap\ZZ^d$ where two nodes $v$
and $u$ are adjacent if $u-v\in\cM$. The set of moves $\cM$ is
\emph{minimal} for $P$ if every move in $\cM$ contributes an edge.
A minimal set of moves $\cM$ is a \emph{Markov basis} of $P$ if
$\fiberGraph{P}{\cM}$ is connected.
\end{defn}

\begin{remark}\label{r:AlgebraicStatistics}
The notions \emph{Markov bases} and \emph{fiber graphs} come from algebraic
statistics~\cite{Drton2008}. The goal there is to run irreducible
Markov chains on \emph{fibers} of a Matrix $A\in\ZZ^{m\times
d}$, i.e. the sets $A^{-1}b:=\{u\in\NN^d: Au=b\}$ for $b\in\ZZ^m$.
With tools from commutative algebra~\cite{Diaconis1998a,Drton2008}, a
universal set of moves $\cM\subset\ker_\ZZ(A)$ can be computed such that the
fiber graphs on all fibers of $A$ are connected simultaneously.
\end{remark}

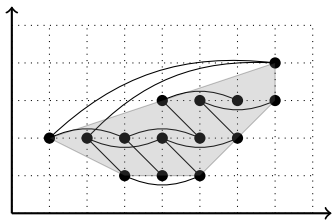
\begin{figure}[h!]
	\begin{tikzpicture}[xscale=0.5,yscale=0.5]
		\node [fill, circle, inner sep=1.5pt](b2) at (1,2) {};
		\node [fill, circle, inner sep=1.5pt](b2) at (2,2) {};
		\node [fill, circle, inner sep=1.5pt](b2) at (3,1) {};
		\node [fill, circle, inner sep=1.5pt](b2) at (3,2) {};
		\node [fill, circle, inner sep=1.5pt](b2) at (4,1) {};
		\node [fill, circle, inner sep=1.5pt](b2) at (4,2) {};
		\node [fill, circle, inner sep=1.5pt](b2) at (5,1) {};
		\node [fill, circle, inner sep=1.5pt](b2) at (5,2) {};
		\node [fill, circle, inner sep=1.5pt](b2) at (5,3) {};
		\node [fill, circle, inner sep=1.5pt](b2) at (6,2) {};
		\node [fill, circle, inner sep=1.5pt](b2) at (6,3) {};
		\node [fill, circle, inner sep=1.5pt](b2) at (7,3) {};
		\node [fill, circle, inner sep=1.5pt](b2) at (7,4) {};
		\node [fill, circle, inner sep=1.5pt](b2) at (4,3) {};
         
         \draw[thick,->](0,0) -- (0,5.5);
         \draw[thick,->](0,0) -- (8.5,0);

      \foreach \X in {0,1,2,3,4,5,6,7,8}{
         \draw[dotted](\X,0) -- (\X,5);
      }

      \foreach \Y in {0,1,2,3,4,5}{
         \draw[dotted](0,\Y) -- (8,\Y);
      }

         \draw(2,2) -- (3,1);
         \draw(3,2) -- (4,1);
         \draw(4,2) -- (5,1);

         \draw(4,3) -- (5,2);
         \draw(5,3) -- (6,2);

          \draw[bend angle=25, bend left](1,2) to (3,2);
          \draw[bend angle=25, bend right](2,2) to (4,2);
          \draw[bend angle=25, bend left](3,2) to (5,2);
          \draw[bend angle=25, bend right](4,2) to (6,2);

          \draw[bend angle=25, bend right](3,1) to (5,1);

          \draw[bend angle=25, bend left](4,3) to (6,3);
          \draw[bend angle=25, bend right](5,3) to (7,3);

          \draw[bend angle=25, bend right](7,4) to (1,2);
          \draw[bend angle=25, bend right](7,4) to (2,2);
          \draw [fill=gray, opacity=0.25] (1,2) --(7,4) -- (7,3) --
(5,1) --(3,1) --cycle; 

	\end{tikzpicture}
	\caption{\label{f:Fibergraph}A fiber graph in $\QQ^2$.}
\end{figure}

\begin{prop}\label{p:EveryGraphIsFiberGraph}
Every simple graph is isomorphic to a fiber graph.
\end{prop}
\begin{proof}
Let $G=(\{v_1,\ldots,v_n\},E)$ be graph and let
$P:=\{x\in\QQ^n_{\ge 0}:
\sum_{i=1}^nx_i=1\}$ be the
$(n-1)$-dimensional simplex, then $P\cap\ZZ^n=\{e_1,\ldots,e_n\}$.
Consider $\cM:=\{e_i-e_j: \{v_i,v_j\}\in E\}$, then
$\fiberGraph{P}{\cM}$ is isomorphic to $G$. 
\end{proof}

The restriction on graphs without loops is necessary, since, in a
fiber graph, either every node has a loop, or none.  The next lemma
states that every fiber graph can be written as a fiber graph in a
full dimensional polytope.

\begin{lemma}\label{l:FullDim}
Let $P\subset\QQ^m$ be a $d$-dimensional polytope and 
$\cM\subset\ZZ^m$ a set of moves. There exists a
$d$-dimensional polytope $P'\subset\QQ^d$ and moves
$\cM'\subset\ZZ^d$ such that~$\fiberGraph{P}{\cM}\cong\fiberGraph{P'}{\cM'}$.
\end{lemma}
\begin{proof}
We can assume that $d<m$. Translation of $P$ does not change the graph
structure of $\fiberGraph{P}{\cM}$ and thus we can assume that
$P\subset\QQ_{\ge 0}^m$. Since $P$ is a rational polytope and since
$d<m$, there exists a matrix $A\in\ZZ^{n\times m}$ with
$\dim\ker_\ZZ(A)=0$ and $n\ge m$ and 
$b\in\ZZ^n$ such that $P=\{x\in\QQ_{\ge 0}^m: Ax\le b\}$. Consider the
injective and affine map
\begin{equation*}
\phi: \QQ^m\to\QQ^{m+n}, x\mapsto\begin{pmatrix}x\\ b-Ax\end{pmatrix}.
\end{equation*}
Then $\phi(P)=\{(x,y)^T\in\QQ_{\ge 0}^{m+n}: Ax+y=b\}$ and
$\dim(\phi(P))=\dim(P)=d$. The set $\cN:=\{(m,Am)^T: m\in\cM\}$ is a
set of moves and the graphs $\fiberGraph{\phi(P)}{\cN}$ and
$\fiberGraph{P}{\cM}$ are isomorphic.  Hence, it suffices to show the
statement for $d$-dimensional polytopes of the form $P=\{x\in\QQ_{\ge
0}^k: Bx=b\}$ for some $b\in\ZZ^n$, a matrix $B\in\ZZ^{n\times k}$,
and a set of moves $\cM\subset\ker_\ZZ(B)$ with the property that
$k\ge n$, $\rank(B)=n$, and $\dim(\ker_\ZZ B )=k-n\ge d$. We can add
rows to $B$ without changing $P$ such that $\dim(\ker_\ZZ B)=k-n= d$.
First, we transform $B$ into its Hermite normal
form, that is, we write $B=(H,0)\cdot C$ for an unimodular matrix
$C\in\ZZ^{k\times k}$ and a matrix $H\in\ZZ^{n\times n}$ of full rank.
Let $H^{-1}\in\QQ^{n\times n}$ and $C^{-1}\in\QQ^{k\times k}$ be the
inverse matrices of $H$ and $C$ respectively. Since $C$ is
unimodular, $C^{-1}\in\ZZ^{k\times k}$ and thus let
$C_1\in\ZZ^{k\times n}$ and $C_2\in\ZZ^{k\times d}$ such that
$C^{-1}=(C_1, C_2)$ and consider the affine map
\begin{equation*}
\psi:\ZZ^d\to\ZZ^k, x\mapsto
C^{-1}\begin{pmatrix}H^{-1}b\\x\end{pmatrix}.
\end{equation*}
Clearly, $\psi$ is
injective and the image of the polytope $P':=\{v\in\QQ^d:
C_2\cdot v\le C_1 H^{-1}b\}\subset\QQ^{d}$ is $P$. Since
$P\cap\ZZ^k\neq\emptyset$, $H^{-1}b\in\ZZ^n$ (see
\cite[Theorem~2.3.6]{Loera2013}) and since $C$ is unimodular, integer
points of $P'$ get mapped to integer points of $P$. Thus
$\dim(P')=\dim(P)=d$. That is,
$P'$ is full dimensional in $\QQ^d$. Consider
\begin{equation*}
\cM':=\{\psi^{-1}(v)-\psi^{-1}(u): v,u\in P\cap\ZZ^k, v-u\in\cM\},
\end{equation*}
then $\cM'=-\cM'$ and $\cM'$ cannot contain multiples.
Let $\psi(v')=v$ and $\psi(u')=u$ for $v',u'\in P'\cap\ZZ^d$, then
$v'-u'\in\cM'$ if and only if $v-u\in\cM$.  Thus, all edges in
$\fiberGraph{P'}{\cM'}$ are mapped bijective to edges in
$\fiberGraph{P}{\cM}$ under $\psi$, which proves that these graphs are
isomorphic.
\end{proof}

\begin{defn}\label{d:FiberDimension}
Let $G$ be a graph. The \emph{fiber dimension} $\fdim{G}$ of $G$ is the smallest
$d\in\NN$ such that there exists a full dimensional lattice polytope
$P\subset\QQ^{d}_{\ge 0}$ and a set of moves $\cM\subset\ZZ^d$ with
$G\cong\fiberGraph{P}{\cM}$. 
\end{defn}

\begin{remark}\label{r:Dimension}
In general, the fiber dimension of a graph $G$ is different from its
\emph{dimension} $\dim(G)$ as defined in~\cite{Erdos1965}.
For example, the complete graph $K_5$ can be realized as fiber graph in
$\QQ^3$ (see Theorem~\ref{t:CompleteGraphs} and
Figure~\ref{f:CompleteGraph}), in contrast to $\dim(K_5)=4$.
\end{remark}

\begin{remark}\label{r:UpperBoundOnfdim}
Proposition~\ref{p:EveryGraphIsFiberGraph} and Lemma~\ref{l:FullDim}
imply that the fiber dimension of any graph with $n$ nodes is bounded
from above by $n-1$. All graphs with at most one vertex have fiber
dimension $0$ and all graphs on at least two nodes without edges have
fiber dimension~$1$.
\end{remark}

\begin{remark}\label{r:ASMarkovBasis}
Let $P$, $\cM$, and $G$ as in the proof of
Proposition~\ref{p:EveryGraphIsFiberGraph} and consider the integer
matrix $A=(1,\dots,1)\in\ZZ^{1\times n}$. If $G$ is connected, than it
is easy to show that $\cM$ is a Markov basis for all polytopes of the
form $\{u\in\QQ^n_{\ge 0}: Au=b\}$ with $b\in\QQ$. In particular,
$\cM$ is a Markov basis of the matrix $A$ in the sense
of~\cite[Definition~1.1.12]{Drton2008} (see also
Remark~\ref{r:AlgebraicStatistics}).
\end{remark}

\begin{prop}\label{p:LinearIndependentMoves}
Let $P\subset\QQ^d$ be a lattice polytope and $\cM\subset\ZZ^d$ be a
Markov basis of $P$ with $2\cdot\dim(P)=|\cM|$, then
$\fiberGraph{P}{\cM}$ is bipartite.
\end{prop}
\begin{proof}
Let $k:=\dim(P)\le d$. Since $\cM$ is a Markov basis of $P$,
$\dim(P)=\dim(\QQ\cdot\cM)$ and thus we can write
$\cM=\{m_1,-m_1,\dots,m_k,-m_k\}$. The assumption on the
dimension says that $\{m_1,\dots,m_k\}$ is linear independent.
Let $v\in P\cap\ZZ^d$ and let
$v+\sum_{i=1}^k\lambda_im_i+\sum_{i=1}^k-\mu_i m_i=v$ be a cycle in
$\fiberGraph{P}{\cM}$ of length $r=\sum_{i=1}^k(\lambda_i+\mu_i)$. The
linear independence gives that $\lambda_i=\mu_i$ for all $i\in[k]$ and
thus $r$ is even.
\end{proof}

\begin{remark}\label{r:ConverseIndepMoves}
The converse of Proposition~\ref{p:LinearIndependentMoves} is false in
general since the $8$-cycle can be minimally embedded in $\QQ^1$
(Proposition~\ref{p:Cycles}) with a Markov
basis consisting of $4$ moves.
\end{remark}

\begin{remark}\label{r:LinearIndependence}
Any Markov basis $\cM\subset\ZZ^d$ of a polytope $P\subset\QQ^d$
fulfills $|\cM|\ge\dim(\QQ\cdot\cM)=\dim(P)$. Thus,
Proposition~\ref{p:LinearIndependentMoves} yields a
lower bound on the number of moves in an embedding of non-bipartite
graphs: If $G$ is a graph with $\chi(G)>2$, then any embedding as a
fiber graph needs strictly more than $2d$ moves.
\end{remark}

\begin{remark}
Every $1$-dimensional lattice polytope in $P\subset\QQ^1$ has a Markov basis of size
$|\cM|=2\dim(P)$, namely $\cM=\{-1,1\}$. This equation fails to be true
already in $\QQ^2$: All Markov bases $\cM\subset\ZZ^2$
of the polytope $P\subset\QQ^2$ shown in
Figure~\ref{f:LinearIndependence} have more than $6$ elements, i.e.
$2\cdot\dim(\QQ\cdot\cM)=4<|\cM|$.
\end{remark}

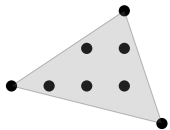
\begin{figure}[htbp]
	\begin{tikzpicture}[xscale=0.5,yscale=0.5]
		\node [fill, circle, inner sep=1.5pt](b2) at (-1,0) {};
		\node [fill, circle, inner sep=1.5pt](b2) at (0,0) {};
		\node [fill, circle, inner sep=1.5pt](b2) at (1,0) {};
		\node [fill, circle, inner sep=1.5pt](b2) at (2,0) {};
		\node [fill, circle, inner sep=1.5pt](b2) at (2,1) {};
		\node [fill, circle, inner sep=1.5pt](b2) at (2,2) {};
		\node [fill, circle, inner sep=1.5pt](b2) at (1,1) {};
		\node [fill, circle, inner sep=1.5pt](b2) at (3,-1) {};
         
\draw [fill=gray, opacity=0.25] (-1,0) --(3,-1) -- (2,2) --cycle; 
	\end{tikzpicture}

\caption{A polytope without a Markov basis with fewer than $6$
moves.}\label{f:LinearIndependence}\end{figure}

\section{Bounds on the fiber dimension}

In this section, we explore upper bounds on the fiber dimension. Our
first observation is that the cartesian products of graphs
behaves nicely with cartesian products of polytopes.

\begin{prop}\label{p:CartesianProduct}
Let $G_1,\ldots,G_n$ be graphs, then $\fdim{\times_{i=1}^n
G_i}\le\sum_{i=1}^n\fdim{G_i}$.
\end{prop}
\begin{proof}
It suffices to prove the inequality for $n=2$. Let $P_1,P_2,\cM_1,\cM_2$
such that $G_i\cong\fiberGraph{P_i}{\cM_i}$. The cartesian product
$P:=P_1\times P_2$ is a polytope of dimension
$\dim(P_1)+\dim(P_2)$.  Additionally, let $\cM:=\{(m,0)^T:
m\in\cM_1\}\cup\{(0,m)^T: m\in\cM_2\}$. It is straight-forward to
check that
$\fiberGraph{P}{\cM}=\fiberGraph{P_1}{\cM_1}\times\fiberGraph{P_2}{\cM_2}$.
Hence, $\fdim{G_1\times G_2}\le \dim(P)$.
\end{proof}

\begin{remark}\label{r:CartesianBoundNotSharp}
The inequality given in Proposition~\ref{p:CartesianProduct} is sharp
for $K_2\times K_2=C_4$ (see Theorem~\ref{t:CompleteGraphs} and
Proposition~\ref{p:Cycles}).
\end{remark}

\begin{prop}\label{p:Apex}
Let $G$ be a graph and $v\in\vertices{G}$, then
$\fdim{G}\le\fdim{G-v}+1$.
\end{prop}
\begin{proof}
Write $\vertices{G}=\{v_0,\dots,v_n\}$ with $v_0:=v$.
Let $d:=\fdim{G-v}$ and let $\phi: G-v\to\fiberGraph{P}{\cM}$ a graph
isomorphism that embeds $G-v$ in dimension $d$ for a
polytope $P\subset\QQ^d$ and a set of moves $\cM\subset\ZZ^d$.
Let $P':=\conv{0,(1,\phi(v_1)),\dots,(1,\phi(v_n))}\subset\QQ^{1+d}$, then
$\dim(P')=d+1$ and $P'\cap\ZZ^{d+1}=\{0,\phi(v_1),\dots,\phi(v_n\}$.
Let $N\subseteq\{v_1,\dots,v_n\}$ be the neighborhood of $v$ in $G$ and
consider $\cM'=\{(0,m): m\in\cM\}\cup\{\phi(v_i): v_i\in N\}$.
Then $G\cong\fiberGraph{P'}{\cM'}$.
\end{proof}

As in~\cite{Erdos1965}, we obtain an upper bound on the dimension in
terms of the chromatic number of the graph.  Our strategy the
following: First, we construct sets of integer points which represent
the color-classes of the graph in such a way that we can freely assign
moves within them. In a second step, we map the vertices of the graph
on these sets and construct the set of moves accordingly. For this
method to work, the constructed integer points must be the integer
points of a polytope. 

\begin{defn}
A finite set $F\subset\ZZ^d$ is \emph{normal} if
$\conv{F}\cap\ZZ^d=F$.
\end{defn}

\begin{thm}\label{t:ChromaticBound}
Let $G$ be a graph with a $k$-coloring in which $r$ color-classes
have cardinality~$1$, then $\fdim{G}\le2\cdot k-r-1$.
\end{thm}
\begin{proof}
Write $\vertices{G}=V_1\cup\cdots\cup V_k$ and set $n_i:=|V_i|$.
Define for $i\in[k]$
\begin{equation*}
W_i:=\{(e_i,j\cdot e_i)^T\in\NN^{2k}: j\in[n_i]\}\subset\NN^{2k}
\end{equation*}
and let $W:=\cup_{i=1}^k W_i$ and $P:=\conv{W}$.
To show that $P\cap\ZZ^{2k}=W$, let $u\in P\cap\ZZ^{2k}$. Since
every $W_i$ is normal, there exists $w_i\in W_i$ and
$\lambda_1,\dots,\lambda_k\in\QQ$ with $0\le\lambda_i\le 1$ for all
$i\in[k]$ and $\sum_{i=1}^k\lambda_i=1$
such that $u=\sum_{i=1}^k\lambda_i w_i$. The projection of $W$
onto the first $k$ coordinates is the set of integer points of the
standard simplex and thus normal. The projection of $u$ onto its first
$k$ coordinates is $e_i$ for some $i\in[k]$. In particular, the only
choice to build an integer vector is thus
$\lambda_j=0$ for $j\neq i$ and $\lambda_i=1$, i.e., $u=w_i\in W_i$.

Let us now construct a graph on $\cP\cap\ZZ^{2k}$ which is isomorphic
to $G$. For that, let $\phi:\cup_{i=1}^kV_i\to\cup_{i=1}^kW_i$
be any bijection which maps elements from $V_i$ to $W_i$ and consider
the set of moves $\cM=\{\phi(v)-\phi(w): \{v,w\}\in\edges{G}\}$. By
construction of $\cM$, $\phi$ is a graph homomorphism from $G$ to
$\fiberGraph{P}{\cM}$. Since edges in $G$ do only connect nodes from
different color classes, the first $k$ coordinates of any element in
$\cM$ do only contain elements from $\{-1,0,1\}$ and thus $\cM$ cannot
contain multiples. Next, let $s\in[n_i]$ and $t\in[n_j]$ such that
$(e_i,se_i)^T-(e_j,te_j)^T=\phi(u)-\phi(v)\in\cM$ with
$v,w\in\vertices{G}$. It follows immediately that
$\phi(v)=(e_j,te_j)^T$ and hence $\phi(u)=(e_i,se_i)^T$. Thus, $\phi$
maps edges from $G$ to $\fiberGraph{P}{\cM}$ bijectively and hence
$\fdim{G}\le\dim(P)$. The vertices of the polytope $P$ are
$\{(e_1,e_1)^T,(e_1,n_1e_1)^T,\dots,(e_k,e_k)^T,(e_k,n_ke_k)^T\}$ and
since $n_i=1$ for $r$ indices $i\in[k]$, $\dim(P)\le 2k-r-1$.  
\end{proof}

\begin{cor}\label{c:ChromaticBound}
For any graph $G$, $\fdim{G}\le2\cdot\chi(G)-1$.
\end{cor}

\begin{remark}\label{r:FourColor}
By the Four Color Theorem~\cite{Appel1977}, the fiber dimension of
every planar graph is at most~$7$.
\end{remark}

\section{Fiber dimension one}

The class of connected graphs of fiber dimension $1$ consists of far
more than path graphs. To see it, let us first specialize
Definition~\ref{d:FiberGraph} to the $1$-dimensional case:

\begin{defn}\label{d:DifferenceGraphs}
Let $n\in\NN_{\ge 1}$ and let $\cD\subseteq[n-1]$ be a finite set such
that for all distinct $d,d'\in\cD$ neither $d|d'$ nor $d'|d$. The graph
$\diffGraph{n}{\cD}$ has nodes $[n]$ where $i$ and $j$ are adjacent if
$|i-j|\in\cD$. A graph $G$ which is isomorphic to $\diffGraph{n}{\cD}$
is a \emph{difference graph}.
\end{defn}

\begin{prop}\label{p:FiberDim1}
A graph has fiber dimension $1$ if and only if it is a
difference graph.
\end{prop}

\begin{lemma}\label{l:Coprime}
Let $n\in\NN$ with $n\ge 2$ and $\cD\subseteq[n]$. If $\diffGraph{n}{\cD}$ is
connected, then $\gcd(\cD)=1$.
\end{lemma}
\begin{proof}
Since $n\ge 2$, there exists a path between $1$ and $2$ in
$\diffGraph{n}{\cD}$. Let $d_1,\dots,d_k\in\cD$ be the distinct
integers that appear in that path and write
$1+\sum_{i=1}^k\lambda_id_i=2$ for
$\lambda_1,\dots,\lambda_k\in\ZZ\setminus\{0\}$. Then
$\gcd(d_1,\dots,d_k)$ divides $1$.
\end{proof}

\begin{prop}\label{p:Cycles}
For any $n\in\NN$ with $n\ge 3$,
\begin{equation*}
\fdim{C_n}=\begin{cases}
1\text{, if }n\not\in\{3,4,6\}\\
2\text{, if }n\in\{3,4,6\}
\end{cases}.
\end{equation*}
\end{prop}
\begin{proof}
Let $n\ge 3$ with $n\in\NN\setminus\{3,4,6\}$. We first show that
there exists an integer $k\in\NN$ with $2\le k<\frac{n}{2}$ such that
$\gcd(k,n)=1$. Let $\phi:\NN\to\NN$ be Euler's totient function. Since
$n\in\NN\setminus\{3,4,6\}$ and $n\ge 3$, $\phi(n)\ge 4$ and we have
for all $k\in[n]$, $\gcd(k,n)=1$ if and only if $\gcd(n-k,k)=1$. In
particular, coprime elements of $n$ come in pairs $(k,n-k)$ with
$k<n-k$. Thus, since $\phi(n)\ge 4$, there must exists $k\in[n]$
with $1<k<\frac{n}{2}$ such that $\gcd(k,n)=1$.  We now show that
$\diffGraph{n}{\{k,n-k\}}$ is a cycle of length $n$.  Clearly,
$n-k$ is not a multiple of $k$ since this would imply that $n$ is
a multiple of $k$ as well which in turn would contradict
$\gcd(n,k)=1$ since $k>1$.  Any node in $\diffGraph{n}{\{k,n-k\}}$
has degree $2$ and hence it suffices to prove that this graph is
connected. Since $k$ and $n$ are coprime, $\langle
k+n\ZZ\rangle=\ZZ_n$. Now, take distinct $i,j\in[n]$, then there
exists $s\in\NN$ such that $j+n\ZZ=i+sk+n\ZZ$ in $\ZZ_n$.  For any
$r\in[s]$, let $i_r\in[n]$ such that $i_r+n\ZZ=i+rk+n\ZZ$. Either
$i_r+k$ or $i_r-(n-k)$ are in $[n]$ and since their congruence
classes in $\ZZ_{n}$ coincide, $i_{r-1}$ and $i_{r}$ are adjacent
in $\diffGraph{n}{\{k,n-k\}}$.  Since $i_k=j$, $i$ and $j$ are
connected. It follows that $C_n=\diffGraph{n}{\{k,n-k\}}$.

Conversely, let $n\in\{3,4,6\}$. Clearly, $\fdim{C_n}\le 2$ due to
Proposition~\ref{p:Apex} since a path has fiber dimension $1$. Hence,
it suffices to show that $\fdim{C_n}>1$ for $n\in\{3,4,6\}$. If $n=3$,
then $C\cong K_3$ and the claim follows from
Theorem~\ref{t:CompleteGraphs}. If $n\in\{4,6\}$, assume that there
exists $\cD\subseteq[n-1]$ such that
$C_n\cong\diffGraph{n}{\cD}$ is a difference graph. It is easy to see
that $|\cD|=2$ since if $|\cD|\ge 3$ or $|\cD|=1$, the node $1$ has
either degree greater than $3$ or is a leaf respectively. Thus, we can
write $\cD=\{d_1,d_2\}$.  Since $\diffGraph{n}{\cD}$ is connected,
$\gcd(d_1,d_2)=1$ by Lemma~\ref{l:Coprime}. Hence, the only possible
choices for $\{d_1,d_2\}$ are $\{2,3\}$ if $n=4$ and
$\{\{2,3\},\{2,5\},\{3,4\},\{3,5\},\{4,5\}\}$ if $n=6$. However, in
all these cases, $\diffGraph{n}{\{d_1,d_2\}}$ is not a cycle.
\end{proof}

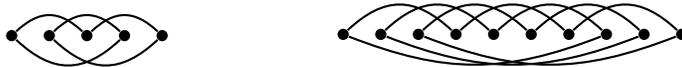
\begin{figure}[h!]
	\begin{tikzpicture}[xscale=0.25,yscale=0.5]
		\node [fill, circle, inner sep=1.5pt](b1) at (0,0) {};
		\node [fill, circle, inner sep=1.5pt](b2) at (2,0) {};
		\node [fill, circle, inner sep=1.5pt](b3) at (4,0) {};
		\node [fill, circle, inner sep=1.5pt](b4) at (6,0) {};
		\node [fill, circle, inner sep=1.5pt](b5) at (8,0) {};
      
      \draw[thick,bend angle=25, bend left](b1) to (b3);
      \draw[thick,bend angle=25, bend left](b3) to (b5);
      \draw[thick,bend angle=25, bend left](b2) to (b4);
      \draw[thick,bend angle=25, bend right](b1) to (b4);
      \draw[thick,bend angle=25, bend right](b2) to (b5);
	\end{tikzpicture}
   \hspace{2cm}
	\begin{tikzpicture}[xscale=0.25,yscale=0.5]
		\node [fill, circle, inner sep=1.5pt](b1) at (0,0) {};
		\node [fill, circle, inner sep=1.5pt](b2) at (2,0) {};
		\node [fill, circle, inner sep=1.5pt](b3) at (4,0) {};
		\node [fill, circle, inner sep=1.5pt](b4) at (6,0) {};
		\node [fill, circle, inner sep=1.5pt](b5) at (8,0) {};
		\node [fill, circle, inner sep=1.5pt](b6) at (10,0) {};
		\node [fill, circle, inner sep=1.5pt](b7) at (12,0) {};
		\node [fill, circle, inner sep=1.5pt](b8) at (14,0) {};
		\node [fill, circle, inner sep=1.5pt](b9) at (16,0) {};
		\node [fill, circle, inner sep=1.5pt](b10) at (18,0) {};
      
      \draw[thick,bend angle=25, bend left](b1) to (b4);
      \draw[thick,bend angle=25, bend left](b4) to (b7);
      \draw[thick,bend angle=25, bend left](b7) to (b10);

      \draw[thick,bend angle=25, bend left](b2) to (b5);
      \draw[thick,bend angle=25, bend left](b5) to (b8);

      \draw[thick,bend angle=25, bend left](b3) to (b6);
      \draw[thick,bend angle=25, bend left](b6) to (b9);

      \draw[thick, bend angle=11,bend right](b1) to (b8);
      \draw[thick, bend angle=11,bend right](b2) to (b9);
      \draw[thick, bend angle=11,bend right](b3) to (b10);

	\end{tikzpicture}\hfill

	\caption{The difference graphs $\diffGraph{5}{\{2,3\}}=C_5$ and
   $\diffGraph{10}{\{3,7\}}=C_{10}$.}

\end{figure}

\section{Complete graphs}

\begin{prop}\label{p:StarGraph}
Let $n\ge 3$, then $\fdim{K_{1,n}}=2$.
\end{prop}
\begin{proof}
Let $v\in\vertices{K_{1,n}}$ be the vertex with maximal degree $n$.
Removing $v$ from $K_{1,n}$ gives a graph on $n\ge 3$ nodes without 
edges, i.e., the fiber dimension of this graph is $1$.
Proposition~\ref{p:Apex} then yields that $\fdim{K_{1,n}}\le 2$.
Conversely, assume that $\fdim{K_{1,n}}=1$ and let $\cD\subset[n]$
such that $K_{1,n}\cong\diffGraph{n+1}{\cD}$ is a difference graph. The
graph isomorphism maps $v$ to some $j\in\{1,\dots,n+1\}$. Since $j$
must be adjacent to all vertices in $\{1,\dots,n+1\}\setminus\{j\}$,
$1\in\cD$. The constraints on $\cD$ imply already that
$\cD=\{1\}$ and thus $\diffGraph{n+1}{\cD}$ is a path.
\end{proof}

\begin{prop}\label{p:CompleteRpartiteGraphs}
For any $n_1,\ldots,n_r\in\NN$, 
$$\fdim{K_{n_1,\ldots,n_r}}\le\lceil\log_2
r\rceil+\lceil\log_2\max\{n_i: i\in[r]\}\rceil.$$
\end{prop}
\begin{proof}
First, decompose the vertex set of  $K_{n_1,\ldots,n_r}$ into its
color clases $V_1,\dots,V_r$ such that $|V_i|=n_i$. Let
$s:=\lceil\log_2 r\rceil$ and $m:=\lceil\log_2\max\{n_i:
i\in[r]\}\rceil$. We prove the upper bound by realizing
$K_{n_1,\ldots,n_r}$ as fiber graph. For any $i\in[r]$, we can choose
a set $W_i\subseteq\{0,1\}^m$ of size $n_i$ since $n_i\le 2^m$.
Similarly, choose a set $C:=\{c_1,\ldots,c_r\}\subseteq\{0,1\}^s$ of
size $r$. The set 
$$\cup_{i=1}^r \{c_i\}\times W_i\subseteq\{0,1\}^{s+m}$$
has cardinality $n_1+\cdots+n_r$ and is normal since all subsets of of
$\{0,1\}^{s+m}$ are normal. Let $\cP$ be its convex hull and let
$\phi:\cup_{i=1}^rV_i\to\cP\cap\NN^{s+m}$ be bijective map which maps
nodes from $V_i$ to $W_i$ (here, it doesn't matter which node of color
$i$ gets mapped to which node in $W_i$ since the colour classes are
independent sets). Let us now construct the Markov basis. Let
$\cM:=\{c_i-c_j: i\neq j\}\times\{-1,0,1\}^m$. The first $s$
coordinates of every element in $\cM$ are non-zero and hence there are
no edges within $\{c_i\}\times W_i$ for any $i\in[r]$ in
$\fiberGraph{\cP}{\cM}$.  Since for distinct $i$ and $j$, all elements
in $\{c_i\}\times W_i$ and $\{c_j\}\times W_j$ are adjacent,
$\fiberGraph{\cP}{\cM}$ is isomorphic to $K_{n_1,\ldots,n_r}$ and with
Lemma~\ref{l:FullDim}, the claim follows.
\end{proof}

\begin{lemma}\label{l:DiscreteBlichfeldt}
Let $\cL\subset\ZZ^n$ be a lattice of full rank and let
$P\subset\QQ^d$ be a set such that for any distinct $v,w\in
P\cap\ZZ^d$, $v-w\not\in\cL$.  Then $|P\cap\ZZ^d|\le|\ZZ^d/\cL|$.
\end{lemma}
\begin{proof}
Let $P\cap\ZZ^d=\{v_1,\dots,v_n\}$ and consider the linear map
$\phi:\ZZ^d\to\ZZ^d/\cL$, $\phi(v)=v+\cL$. By assumption,
$\phi(v_i-v_j)\neq 0$ in $\ZZ^d/\cL$ for all $i,j\in[n]$ with $i\neq
j$. Assume that there are $i,j\in[n-1]$ with $i\neq j$ such that
$\phi(v_n-v_i)=\phi(v_n-v_j)$. Then
$\phi(v_i-v_j)=\phi(v_i-v_n+v_n-v_j)=\phi(v_i-v_n)-\phi(v_j-v_n)=0$, a
contradiction. Thus, $\phi(v_n-v_i)\neq\phi(v_n-v_j)$ for all
distinct $i,j\in[n-1]$.  That is, $|\{\phi(v_n-v_i):i\in[n-1]\}|=n-1$.
The proposition follows from $n-1=|\{\phi(v_n-v_i):
i\in[n-1]\}|\le|\ZZ^d/\cL|-1$
\end{proof}

\begin{remark}\label{r:Blichfeldt}
Since $|\ZZ^n/\cL|=\det(\cL)$, Lemma~\ref{l:DiscreteBlichfeldt}
can be seen as a discrete analogon of Blichfeldt's
theorem~\cite[Theorem~2.4.1]{Loera2013}.
\end{remark}

\begin{thm}\label{t:CompleteGraphs}
For any $n\in\NN$, $\fdim{K_n}=\lceil\log_2n\rceil$.
\end{thm}
\begin{proof}
Due to Proposition~\ref{p:CompleteRpartiteGraphs}, the fiber dimension
of $K_n$ is bounded from above by $\lceil\log_2n\rceil$. Let
$d:=\fdim{K_n}$ and $P\subset\QQ^d$ a $d$-dimensional polytope and
$\cM\subset\ZZ^d$ a Markov basis such that
$K_n\cong\fiberGraph{P}{\cM}$. Assume there are $v,w\in P\cap\ZZ^d$
such that $v-w\in 2\cdot\ZZ^d$. Since $(v+w)_i$ is even for all
$i\in[d]$, $(v+w)_i$ is even for all $i\in[d]$ and thus $v+w\in
2\ZZ^d$. In particular, then $\frac{1}{2}(v+w)\in\ZZ^d$ and
since $P$ is normal, $\frac{1}{2}(v+w)\in P\cap\ZZ^d$. This, however,
implies that $v-w\in\cM$ and $\frac{1}{2}(v-w)\in\cM$. Thus,
$v-w\not\in 2\ZZ^d$. Due to Lemma~\ref{l:DiscreteBlichfeldt},
$n=|P\cap\ZZ^d|\le 2^d$ and thus $d\ge\lceil\log_2 n\rceil$.
\end{proof}

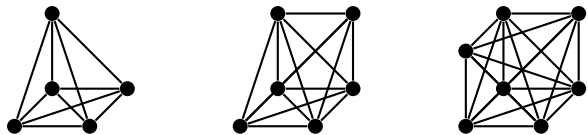
\begin{figure}[h!]
	\begin{tikzpicture}[xscale=0.5,yscale=0.5]

		\node [fill, circle, inner sep=2pt](a1) at (-1,-1) {};
		\node [fill, circle, inner sep=2pt](a2) at (1,-1) {};
		\node [fill, circle, inner sep=2pt](b1) at (0,0) {};
		\node [fill, circle, inner sep=2pt](b2) at (2,0) {};
		\node [fill, circle, inner sep=2pt](b3) at (0,2) {};

      \draw[thick](a1) -- (a2);
      \draw[thick](a1) -- (b1);
      \draw[thick](a1) -- (b2);
      \draw[thick](a1) -- (b3);
      \draw[thick](a2) -- (b1);
      \draw[thick](a2) -- (b2);
      \draw[thick](a2) -- (b3);
      \draw[thick](b1) -- (b2);
      \draw[thick](b3) -- (b2);
      \draw[thick](b3) -- (b1);

		\node [fill, circle, inner sep=2pt](d1) at (6,0) {};
		\node [fill, circle, inner sep=2pt](d2) at (8,0) {};
		\node [fill, circle, inner sep=2pt](d3) at (6,2) {};
		\node [fill, circle, inner sep=2pt](d4) at (8,2) {};
		\node [fill, circle, inner sep=2pt](d5) at (5,-1) {};
		\node [fill, circle, inner sep=2pt](d6) at (7,-1) {};

     \draw[thick](d1) --(d2); 
     \draw[thick](d1) --(d3); 
     \draw[thick](d1) --(d4); 
     \draw[thick](d1) --(d5); 
     \draw[thick](d1) --(d6); 
       
     \draw[thick](d2) --(d3); 
     \draw[thick](d2) --(d4); 
     \draw[thick](d2) --(d5); 
     \draw[thick](d2) --(d6); 

     \draw[thick](d3) --(d4); 
     \draw[thick](d3) --(d5); 
     \draw[thick](d3) --(d6); 

     \draw[thick](d4) --(d5); 
     \draw[thick](d4) --(d6); 

     \draw[thick](d5) --(d6); 

		\node [fill, circle, inner sep=2pt](e1) at (12,0) {};
		\node [fill, circle, inner sep=2pt](e2) at (14,0) {};
		\node [fill, circle, inner sep=2pt](e3) at (12,2) {};
		\node [fill, circle, inner sep=2pt](e4) at (14,2) {};
		\node [fill, circle, inner sep=2pt](e5) at (11,-1) {};
		\node [fill, circle, inner sep=2pt](e6) at (13,-1) {};
		\node [fill, circle, inner sep=2pt](e7) at (11,1) {};

     \draw[thick](e1) --(e2); 
     \draw[thick](e1) --(e3); 
     \draw[thick](e1) --(e4); 
     \draw[thick](e1) --(e5); 
     \draw[thick](e1) --(e6); 
     \draw[thick](e1) --(e7); 

     \draw[thick](e2) --(e3); 
     \draw[thick](e2) --(e4); 
     \draw[thick](e2) --(e5); 
     \draw[thick](e2) --(e6); 
     \draw[thick](e2) --(e7); 

     \draw[thick](e3) --(e4); 
     \draw[thick](e3) --(e5); 
     \draw[thick](e3) --(e6); 
     \draw[thick](e3) --(e7); 

     \draw[thick](e4) --(e5); 
     \draw[thick](e4) --(e6); 
     \draw[thick](e4) --(e7); 

     \draw[thick](e5) --(e6); 
     \draw[thick](e5) --(e7); 

     \draw[thick](e6) --(e7); 

	\end{tikzpicture}
   \caption{\label{f:CompleteGraph}Fiber graph embeddings of $K_5,K_6$, and
   $K_7$ in $\QQ^3$.}
\end{figure}

\begin{remark}\label{r:CompleteGraphs}
Proposition~\ref{p:Apex} yields the trivial upper bound for complete
graphs, that is $\fdim{K_n}\le n-1$ for
$n\ge 2$. This bound is strict for the first time for $n=4$.
\end{remark}

\section{Distinct pair-sum polytopes}
For the remainder, we investigate an universal upper
bound on the fiber dimension by generalizing our embedding in
Proposition~\ref{p:EveryGraphIsFiberGraph} into the simplex. A priori,
a move in a Markov basis give rise to distinguished edges in a fiber
graph. The next definition states a property of polytopes assuring
that Markov moves lead to precisely one edge in the graph.

\begin{defn}\label{d:DPSPoly}
A lattice polytope $P\subset\QQ^d$ with $n:=|P\cap\ZZ^d|$ is a
\emph{distinct pair-sum polytope} if
$|P\cap\ZZ^d+P\cap\ZZ^d|=\binom{n}{2}+n$.
\end{defn}

\begin{remark}
Let $P\subset\QQ^d$ be a distinct pair-sum polytope and write
$P\cap\ZZ^d=\{v_1,\dots,v_n\}$, then all the possible sums
$2v_1,\dots,2v_n,v_1+v_2,v_1+v_3,\dots,v_{n-1}+v_n$ are pairwise distinct. We
refer to~\cite{Choi2002,Curcic2013} for more on distinct pair-sum polytopes.
\end{remark}

The next proposition states that distinct pair-sum polytopes allow
embeddings of all possible graphs whose number of nodes equals the
number of lattice points of the polytope.

\begin{prop}\label{p:DPSBound}
Let $P\subset\QQ^d$ be a distinct pair-sum polytope with
$n:=|P\cap\ZZ^d|$. For any graph $G$ on $n$ nodes, there exists a set
of moves $\cM\subset\ZZ^d$ such that $G\cong P(\cM)$. 
\end{prop}
\begin{proof}
Pick an arbitrary bijection $\phi:\vertices{G}\to
P\cap\ZZ^n$ and define
$$\cM:=\{\phi(u)-\phi(v): u\text{ and }v\text{ adjacent in }G\}.$$ 
We claim that $G\cong\fiberGraph{P}{\cM}$.  First, we need to show
that $\cM$ does not contain multiples. Assume, there are $m,m'\in\cM$
and $k\in\NN$ with $k\ge 2$ such that $m=km'$.  Let $v,w\in
P\cap\ZZ^d$ with $v-w=m$.  Then $w+km'=v$. Then $w,w+m',w+2m'\in
P\cap\ZZ^d$ are disjoint elements that fulfill
$(w+m')+(w+m')=w+(w+2m')$, that is, $|P\cap\ZZ^d+P\cap\ZZ^d|<n(n-1)+1$
since two different sums lead to the same element. Clearly, every edge
in $G$ get mapped to an edge in $P(\cM)$. Conversely, let $v,w\in
P\cap\ZZ^d$ such that $v-w\in\cM$.  Then there exists adjacent nodes
$v',w'\in\vertices{G}$ with $\phi(v')-\phi(w')=v-w$.  We have to prove
that $\phi(v')=v$ and $\phi(w')=w$. If not, then
$\phi(v')+w=\phi(w')+v$ implies that two different sums yield the same
element in $P\cap\ZZ^d+P\cap\ZZ^d$ which again gives a contraction.
\end{proof}

In~\cite{Choi2002}, a distinct pair-sum polytope in $\QQ^n$ on
$2^n$ lattice points was constructed for any $n\in\NN$. This gives
rise to the following result.

\begin{prop}\label{p:LogDPS}
Let $G$ be a graph on $2^n$ nodes, then $\fdim{G}\le n$.
\end{prop}
\begin{proof}
This is~\cite[Theorem~3]{Choi2002} together with
Proposition~\ref{p:DPSBound}.
\end{proof}

Lower bounds on the fiber dimension can be translated to relations
between the number of lattice points and the dimension of the ambient
space of distinct pair-sum polytopes. The next proposition
demonstrates this for complete graphs and rediscovers a bound which
was already proven in~\cite[Theorem~2]{Choi2002}. 

\begin{prop}\label{p:RelationDPS}
Let $P\subset\QQ^d$ be a distinct pair-sum polytope, then
$|P\cap\QQ^d|\le 2^d$.
\end{prop}
\begin{proof}
Let $n:=|P\cap\QQ^d|$. According to Proposition~\ref{p:DPSBound},
there exists a Markov basis $\cM\subset\ZZ^d$ such that $K_n\cong
P(\cM)$. By the definition of the fiber dimension and
Theorem~\ref{t:CompleteGraphs}, $\lceil\log_2 n \rceil=\fdim{K_n}\le
d$, i.e., $n\le 2^d$.
\end{proof}

\begin{remark}\label{r:SmallestDPSEmbedding}
For any $n\in\NN$, there exists a distinct pair-sum
polytope on $n$ lattice points, for example, take the $(n-1)$
dimensional simplex in $\QQ^{n-1}$. Thus, for fixed $n\in\NN$, we can
ask for the smallest natural number $d\in\NN$ such that there exists a
distinct pair-sum polytope $P\subset\QQ^d$ on $n$ lattice points. Then
$d\le n-1$ and Proposition~\ref{p:RelationDPS} on the other hand gives
$\lceil\log_2n\rceil\le d$. Given such a minimal $d$,
Proposition~\ref{p:DPSBound} implies the fiber dimension of any graph
on $n$ nodes is bounded from above by $d$. However, the embedding of
graphs into distinct pair-sum polytopes is far from optimal. 
For instance, a path has fiber dimension $1$ and thus this bound can
be made arbitrarily bad.  However, we think its an interesting
question for which class beside complete graphs this bound is
(asymptotically) tight.
\end{remark}

\bibliographystyle{amsplain}
\bibliography{fiberDimension.bib}

\end{document}